\documentclass[a4paper,11pt,leqno]{article}
\usepackage{amsmath, amsfonts, amssymb, amsthm}
\numberwithin{equation}{section}
\usepackage{graphicx}
\usepackage{float}

\numberwithin{equation}{section}

\newtheorem{theorem}{Theorem}[section]
\newtheorem*{theorem*}{Theorem}
\newtheorem{definition}{Definition}[section]
\newtheorem{corollary}{Corollary}[section]
\newtheorem{proposition}{Proposition}[section]
\newtheorem{lemma}{Lemma}[section]
\newtheorem{remark}{Remark}[section]

\begin{document}

\title{The energy conservation and regularity for the Navier-Stokes equations}

\author{ \footnote{Email Address: tanwenkeybfq@163.com (W. Tan), \quad  mcsyzy@mail.sysu.edu.cn (Z. Yin).}
Wenke $\mbox{Tan}^1$  \quad and \quad Zhaoyang $\mbox{Yin}^{2,3}$ \\
 \small $^1\mbox{Key}$ Laboratory of Computing and Stochastic Mathematics (Ministry of Education), School of \\\small Mathematics and Statistics, Hunan Normal University, Changsha, Hunan 410081, P. R. China  \\ \small $^2\mbox{Department}$ of Mathematics, Sun Yat-sen University, Guangzhou, 510275, China\\
\small $^3\mbox{Faculty}$ of Information Technology,
Macau University of Science and Technology, Macau, China}

\date{}
\maketitle

\begin{abstract}
In this paper, we consider the energy conservation and regularity of the weak solution $u$ to the Navier-Stokes equations in the endpoint case. We first construct a divergence-free field $u(t,x)$ which satisfies $\lim_{t\to T}\sqrt{T-t}||u(t)||_{BMO}<\infty$ and $\lim_{t\to T}\sqrt{T-t}||u(t)||_{L^\infty}=\infty$ to demonstrate that the Type II singularity is admissible in the endpoint case $u\in L^{2,\infty}(BMO)$. Secondly, we prove that if a suitable weak solution $u(t,x)$ satisfying $||u||_{L^{2,\infty}([0,T];BMO(\Omega))}<\infty$ for arbitrary $\Omega\subseteq\mathbb{R}^3$ then the local energy equality is valid on $[0,T]\times\Omega$. As a corollary, we also prove $||u||_{L^{2,\infty}([0,T];BMO(\mathbb{R}^3))}<\infty$ implies the global energy equality on $[0,T]$. Thirdly, we show that as the solution $u$ approaches a finite blowup time $T$, the norm $||u(t)||_{BMO}$ must blow up at a rate faster than $\frac{c}{\sqrt{T-t}}$ with some absolute constant $c>0$. Furthermore, we prove that if $||u_3||_{L^{2,\infty}([0,T];BMO(\mathbb{R}^3))}=M<\infty$ then there exists a small constant $c_M$ depended on $M$ such that if $||u_h||_{L^{2,\infty}([0,T];BMO(\mathbb{R}^3))}\leq c_M$ then $u$ is regular on $(0,T]\times\mathbb{R}^3$.
\end{abstract}


\section{Introduction}
~~~~We consider the incompressible Navier-Stokes in the domain $\Omega\subseteq\mathbb{R}^3$
\begin{equation}\label{NS}
 \left\{\begin{array}{ll}
\partial_tu-\Delta u+u\cdot\nabla u+\nabla P=0,\\
\nabla\cdot u=0,\\
u(x,t)|_{\partial\Omega}=0,\\
u(0,x)=u_0(x)
\end{array}\right.
\end{equation}
where the unknowns $u$, $P$ denote the velocity vector field, pressure respectively.

It is well-known that if $u_0$ is smooth enough, then problems \eqref{NS} have a unique solution on $[0,T)$ for some $T>0$; see, for example, \cite{Leray,Fujita,Kaniel,Kiselev,Wahl} and the references therein.
The concepts of weak solutions of \eqref{NS} and their regularity were already introduced in the fundamental paper of Leray \cite{Leray}. Pioneering works of Leray \cite{Leray} and Hopf \cite{Hopf} showed the global existence of a weak solution called Leray-Hopf solution. In three dimensions, however, the question of regularity and uniqueness of weak solutions is an outstanding open problem in mathematical fluid mechanics.

On the one hand, in Leray's paper \cite{Leray}, he proved that if $[0,T)$ is the maximal existence interval of a smooth solution, then for $p > 3$, there exists $c_p>0$ such that
\begin{align*}
||u(t,\cdot)||_{L^p(\mathbb{R}^3)}\geq\frac{c_p}{|T-t|^{\frac{p-3}{2p}}}.
\end{align*} This means that as the solution $u$ approaches a finite blowup time $T$, the norm $||u(t)||_{L^p}$ with $3<p$ must blow up at a rate $\frac{c_p}{|T-t|^{\frac{p-3}{2p}}}$ with some absolute constant $c_p>0$. In general, if u satisfies
\begin{align*}
||u(t,\cdot)||_{L^p(\mathbb{R}^3)}\leq \frac{C}{|T-t|^{\frac{p-3}{2p}}},
\end{align*}the regularity of the solution at $t = T$ remains unknown. The well-known Ladyzhenskaya-Prodi-Serrin criteria \cite{Ladyzhenskaya2,Prodi,Serrin1}showed that if $u\in L^r([0,T],L^s(\mathbb{R}^3))$ for $\frac{2}{r}+\frac{3}{s}\leq1,s>3$ then $u$ is regular on $[0,T]$. For the endpoint case $p =3$, Escauriaza, Serengin and Sverak \cite{Escauriaza} proved that the $L^\infty L^3$ solutions are smooth. This result was improved by Tao \cite{Tao} showed that as the solution $u$ approaches a finite blowup time $T$, the critical norm $||u(t)||_{L^3}$ must blow up at a rate $(\log\log\log\frac{1}{T-t})^{c}$ with some absolute constant $c>0$. The other endpoint case $p=\infty$ was improved by Kozono and Taniuchi \cite{Kozono} proved that $u\in L^2([0,T];BMO(\mathbb{R}^3))$ implies the regularity of the solution $u$ to \eqref{NS}. Since the condition $||u(t)||_{L^p(\mathbb{R}^3)}\leq\frac{c_p}{|T-t|^\frac{p-3}{2p}}$ only requires $u\in L^{q,\infty}(L^p)$ for $\frac{3}{p}+\frac{2}{q}=1,p\geq3$, it is natural to generalize the classical Ladyzhenskaya-Prodi-Serrin type criterion into Lorentz spaces. In \cite{Kozono1}, Kim and Kozono proved the local boundedness of a weak solution $u$ under the assumption that $||u||_{L^{r,\infty}([0,T];L^{s,\infty}(\mathbb{R}^3))}$ is sufficiently small for some $(r,s)$ with $\frac{2}{r}+\frac{3}{s}=1$ and $3\leq s<\infty$. The limiting case of the regularity criteria derived by Kim and Kozono was proved by He and Wang \cite{He} i.e. any weak solution $u$ to the Navier-Stokes equations is regular under the assumption that $||u||_{L^{2,\infty}([0,T];L^\infty(\mathbb{R}^3))}$ is sufficiently small. This results were improved by Wang and Zhang \cite{Wang} which showed that $||u_3||_{L^{r,\infty}([0,T];L^{s,\infty}(\mathbb{R}^3))}\leq M$ and $||u_h||_{L^{r,\infty}([0,T];L^{s,\infty}(\mathbb{R}^3))}\leq c_M$ with $\frac{2}{r}+\frac{3}{s}=1$ and $3<s\leq\infty$ imply the regularity of the suitable weak solution $u$ to Navier-Stokes equations, where $u_h=(u_1,u_2)$, $u =(u_h, u_3)$ and $c_M$ is a small constant depending on $M$. Collecting the results of \cite{Kozono,He,Wang}, one should notice that the regularity in the endpoint case $u\in L^{2,\infty}BMO$ is left open. It is worth pointing out that although the spaces $L^{2,\infty}L^\infty$ and $L^{2,\infty}BMO$ are on the same scale under the scaling \eqref{scaling}, but we can construct a divergence free field $u(t,x)$ which satisfies $\lim_{t\to T}\sqrt{T-t}||u(t)||_{BMO}<\infty$ and $\lim_{t\to T}\sqrt{T-t}||u(t)||_{L^\infty}=\infty$ (see Proposition \ref{prop}). It is surprising to some extent because the conditions $\lim_{t\to T}\sqrt{T-t}||u(t)||_{BMO}<\infty$ and $\lim_{t\to T}\sqrt{T-t}||u(t)||_{L^\infty}=\infty$ implies that some Type II blow-up are admissible in the invariant space $L^{2,\infty}BMO$.
Inspiring by the above observation, it is of interest to investigate the blow-up rate for $||u(t)||_{BMO}$ as the solution $u$ approaches a finite blow-up time $T$.  This is one of our goals in this
paper

On the other hand, it is well-known that the Leray-Hopf weak solutions are weak continuous in $L^2$ space, but the strong continuity of $u$ in $L^2$ space is still an open problem in the mathematic theory of Navier-Stokes equations. A sufficient condition for strong continuity of $u$ in $L^2$ space is the energy equality,  as should be expected from the physical point of view. The question of energy equality has of course also been extensively studied for Leray-Hopf solutions of the 3-dimensional Navier-Stokes equations. Lions \cite{Lions} and Ladyzhenskaya \cite{Ladyzhenskaya1} proved independently that such solutions satisfy the (global) energy equality under the additional assumption $u\in L^4L^4$. Shinbrot in \cite{Shinbrot} proved the energy equality under the extrapolated version of the Lions-Ladyzh\u{e}nskaya condition, namely $\frac{2}{r}+\frac{2}{s}\leq 1$ for $s\geq 4$. The endpoint case $u\in L^2L^\infty$ was generalized to $u\in L^2(BMO)$ by Kozono and Taniuchi \cite{Kozono}. Kukavica \cite{Kukavica} proved sufficiency of the weaker but dimensionally equivalent criterion $P\in L^2L^2$. Cheskidov, Friedlander and Shvydkoy \cite{Cheskidov} proved energy equality for $u\in L^3D(A^\frac{5}{12})$ on a bounded domain; an extension to exterior domains was proved in \cite{Farwig} by Farwig and Taniuchi. (Here A denotes the Stokes
operator.) Seregin and \u{S}ver\'{a}k \cite{Seregin1} proved energy equality (regularity, in fact) for suitable weak
solutions whose associated pressure is bounded from below in some sense; this paper makes use of the low-dimensionality of the singular set for suitable weak solutions that is guaranteed by the celebrated
Caffarelli-Kohn-Nirenberg Theorem \cite{Caffarelli}. Leslie and Shvydkoy \cite{Leslie1}; Shvydkoy \cite{Shvydkoy}  proved energy equality under assumptions on the size and (or structure of the singularity set)
in addition to the integrability of the solution. Noticing the gap between the energy conservation criteria $\frac{1}{q}+\frac{1}{p}=\frac{1}{2}$ and the regularity criteria $\frac{2}{q}+\frac{3}{p}=1$, it is natural to consider the energy conservation at blow-up time. Recently, Leslie and Shvydkoy \cite{Leslie} proved that any solution to the 3-dimensional Navier-Stokes Equations in $\mathbb{R}^3$ which is Type I (i.e. $||u(t)||_{L^\infty}\leq\frac{C}{\sqrt{T-t}}$) in time must satisfy the energy equality at the first blowup time $T$. This result was improved by Cheskidov and Luo \cite{Luo} by showing $||u(t)||_{L^{2,\infty}(B^{0}_{\infty,\infty})}<\infty$ implies energy equality. Based the example constructed in Proposition \ref{prop}, we thus consider the energy conservation at some potential Type II blow-up time $T$ satisfying $\limsup_{t\to T^-}\sqrt{T-t}||u(t)||_{BMO}<\infty$ and $\limsup_{t\to T^-}\sqrt{T-t}||u(t)||_{L^\infty}=\infty$.

The tool that we use to address the question of energy conservation is the energy measure which was first introduced in \cite{Shvydkoy1}. The potential failure of energy equality for a solution $u$ of \eqref{NS} can be quantified using a so-called energy measure $\mathcal{E}$, which was defined to be the weak-$\ast$ limit of the measures $|u(t)|^2dx$ as $t$ approaches the first possible blowup time. Precisely, assume $[0,T)$ is the maximal existence interval of a smooth solution to $\eqref{NS}$. It is clear that $|u(t)|^2dx$ is a bounded sequence of Radon measures, so that when $t_k\rightarrow T^-$ there exists a subsequence $|u(t_k)|^2dx$ which converges weak-$\ast$ to some Radon measure $\mathcal{E}$ which is called the energy measure at time $T$.

We introduce the following two quantities, the lower local dimension $d(x,\mathcal{E})$ of $\mathcal{E}$ at $x\in\Omega$, and the concentration dimension $D$ of $\mathcal{E}$ in $\Omega$, defined respectively by
\begin{align}
&d(x,\mathcal{E})=\liminf_{r\rightarrow 0^+}\frac{\ln\mathcal{E}(B_r(x))}{\ln r},\\
&D=\inf\{dim_{\mathcal{H}}(S):S\subset\Omega~compact,~and~ \mathcal{E}(S)>0\},
\end{align}with the convention that $D =3$ if the collection over which the infimum is taken is
empty. The local dimension is
a standard geometric measure theoretic quantity, see \cite{Mattila}, while the concentration dimension which  assigns  a numerical value to the concentration of the energy measure, namely the smallest Hausdorff dimension of a set of positive $\mathcal{E}$-measure, was first introduced in \cite{Shvydkoy1}, together with the energy measure itself.

It is worth pointing out that although the definitions of $d(x,\mathcal{E})$ and $D$ are local, the results in \cite{Leslie} require the solution $u$ must be defined on whole space $\mathbb{R}^3$. Hence, the methods which are dependent on a expanding iterative arguments developed in \cite{Leslie} can not deal with the energy measure on the domains with boundary. The methods developed in \cite{Luo} have similar limitations. In order to consider energy measure and the energy conservation in arbitrary domain $\Omega\subseteq\mathbb{R}^3$, we establish a new iterative method to deal with the energy measure and the local energy equality. Our methods  are applicable to other systems for divergence free fields, in particular, to the three-dimensional MHD equations \cite{Tan-Wu}.
\subsection{Definitions and Notations}
\par
~~~~Before presenting the main results, we give some notations and definitions. We first recall the definitions of weak and suitable weak solutions of \eqref{NS}.
\begin{definition}
Let $T>0$. The function $u$ is called a Leray-Hopf weak solution of \eqref{NS} in
$[0,T]\times\Omega$ if\\
1.$u\in L^\infty(0,T; L^2(\Omega))\cap L^2(0,T; H^1(\Omega)$;\\
2.There exists a distribution $P$ such that $(u,P)$ satisfies \eqref{NS} in the sense of distributions;\\
3.$u$ satisfies the energy inequality: for a.e. $t_0\in [0,t]$ including $t_0=0$,
\begin{align}\label{energy inequality}
\int_{\Omega}|u(t,x)|^2dx+2\int_0^t\int_{\Omega}|\nabla u|^2dxdt\leq\int_{\Omega}|u(t_0,x)|^2dx.
\end{align}
\end{definition}

Furthermore, the pair $(u,P)$ is called a suitable weak solution to \eqref{NS} if $u\in L^3([0,T];L^3(\Omega)),P\in L^\frac{3}{2}([0,T]\times\Omega)$ and the following
local energy inequality holds
\begin{align}\label{le}
\int_{\Omega}|u(t,x)|^2\phi dx+2\int_0^t\int_{\Omega}|\nabla u|^2\phi dxdt\leq \int_0^t\int_{\Omega}|u|^2(\partial_t\phi+\Delta\phi)+(|u|^2+2P)u\cdot\nabla\phi dxdt
\end{align} for every nonnegative $\phi\in C^\infty_0((0,T]\times\Omega)$.

 We say a point $z_0=(t_0,x_0)\in[0,T]\times\Omega$ is a regular point of solution $u$ to \eqref{NS} if there exists a non-empty neighborhood $\mathcal{O}_{z_0}\subset[0,T]\times\Omega$ of $z_0$ such that $u\in L^\infty(\mathcal{O}_{z_0})$. The complement of the set of regular points will be called the singular set.

Let $(u,P)$ be a solution of \eqref{NS}. Introduce the scaling
\begin{align}\label{scaling}
u_\lambda(t,x)=\lambda u(\lambda^2t,\lambda x);~~P_\lambda(t,x)=\lambda^2P(\lambda^2t,\lambda x),
\end{align}for arbitrary $\lambda>0$. Then the family $(u,P)$ is also a solution of \eqref{NS}.
Setting
\begin{align*}
B_r(x_0)=\{x\in\mathbb{R}^3:|x-x_0|<r\},B_r=B_r(0),~B=B_1,\\
Q_r(z_0)=B_r(x_0)\times(t_0-r^2,t_0),~Q_r=Q_r(0),~Q=Q_1.
\end{align*}
 We introduce
the following invariant quantities, which are invariant under the natural scaling \eqref{scaling}:
\begin{align*}
A(u,r,z)=\sup_{t-r^2\leq s\leq t}\frac{1}{r}\int_{B_r(x)\times \{s\}}|u|^2dx;~ B(u,r,z)=\frac{1}{r}\int\int_{Q_r(z)}|\nabla u|^2dxdt,\\
C(u,r,z)=\frac{1}{r^2}\int\int_{Q_r(z)}|v|^3dxdt;~ D(P,r,z)=\frac{1}{r^2}\int\int_{Q_r(z)}|P|^{\frac{3}{2}}dxdt.
\end{align*}
For simplicity, we introduce the notations
\begin{align*}
A(u,r)=A(u,r,0);~B(u,r)=&B(u,r,0);~C(u,r)=C(u,r,0);~D(P,r)=D(P,r,0).
\end{align*}

Let $\Omega\subset\mathbb{R}^n$ and $\Omega_T=[0,T]\times\Omega$. we use $L^qL^p(\Omega_T)$ to denote the space of measurable functions with the following norm
\begin{align*}
||f||_{L^qL^p(\Omega_T)}=\left\{\begin{array}{ll}
(\int_0^T(\int_\Omega|f(t,x)|^pdx)^{\frac{q}{p}}dt)^{\frac{1}{q}},~1\leq q<\infty,\\
ess\sup_{t\in[0,T]}||f(t,\cdot)||_{L^p(\Omega)},~q=\infty.
\end{array}\right.
\end{align*}

The Lorentz space $L^{r,s}(\Omega)$ is the space of measurable functions with the following norm:
\begin{align*}
||f||_{L^{r,s}(\Omega)}=\left\{\begin{array}{ll}
(\int_0^\infty\sigma^{s-1}|\{x\in\Omega:|f(x)|>\sigma\}|^{\frac{s}{r}}d\sigma)^\frac{1}{s},~1\leq s<\infty,\\
\sup_{\sigma>0}\sigma|\{x\in\Omega:|f(x)|>\sigma\}|^\frac{1}{r},~s=\infty.
\end{array}\right.
\end{align*}

We say that a local integrable function $f$ is in $BMO(\mathbb{R}^3)$ if it satisfies
\begin{align*}
\sup_{R>0,x_0\in\mathbb{R}^3}\frac{1}{|B_R(x_0)|}\int_{B_R(x_0)}
|f(x)-[f]_{B_R(x_0)}|dx<\infty.
\end{align*} A local integrable function $f$ is in $BMO_{loc}(\mathbb{R}^3)$ if $f\in BMO(\Omega')$ for any $\Omega'\subset\subset\mathbb{R}^3$.
Moreover, $f(t,x)$ is in $L^{2,\infty}([0,T];BMO)$ if $||f(t,\cdot)||_{BMO}\in L^{2,\infty}([0,T])$.

Throughout this paper, $C$ is a positive constant and $C(A1,A2,...)$ is a positive
constant depending on $A1,A2,\cdots$.
\subsection{Main results}
Before the statements of our main results, we show some examples to explain that some potential type II singularity are admissible in the space $L^{2,\infty}([0,T];BMO(\mathbb{R}^3))$.
\begin{proposition}\label{prop}
Let $u(x,t)=\mathbb{P}\{\frac{1}{\sqrt{T-t}}|\frac{x}{\sqrt{T-t}}|^{\sqrt{T-t}}(-\ln|\frac{x}{\sqrt{T-t}}|)^s\phi(\frac{x}{\sqrt{T-t}})\}$, $0<s<\frac{2}{3}$. Here $\mathbb{P}$ is the Leray projection and $\phi$ is a radial cut-off function which equals 1 in $B_{\frac{1}{2}}$ and vanishes outside of $B_{\frac{3}{4}}$. Then, it holds
\begin{align}
||u(t)||^2_{L^2}\sim\sqrt{T-t},&~||\nabla u(t)||^2_{L^2}\sim\frac{1}{\sqrt{T-t}},\notag\\
||u(t)||_{L^\infty}\geq\frac{1}{(T-t)^{1+\frac{s}{2}}},&~||u(t)||_{BMO}\leq\frac{1}{\sqrt{T-t}}.\notag
\end{align}  In particular, it is holding that
\begin{align*}
\lim_{t\to T}\sqrt{T-t}||u(t)||_{(BMO(\mathbb{R}^3))}<\infty~and~\lim_{t\to T}\sqrt{T-t}||u(t)||_{L^\infty}=\infty.
\end{align*}
\end{proposition} Noticing that $||u(t)||_{BMO}\leq\frac{C}{\sqrt{T-t}}$ implies $u\in L^{2,\infty}(BMO)$, we choose naturally the space $L^{2,\infty}(BMO)$ as the workspace.
We now state our main results.
\begin{theorem}\label{main1}
Let $(u,P)$ be a suitable weak solution to \eqref{NS} on $[0,T]\times\Omega$ and $||u||_{L^{2,\infty}([0,T];BMO(\Omega))}=M<\infty$. Then it is holding that
\begin{align}
&\sup_{t-r^2\leq s<t}\frac{1}{r}\int_{B_r(x)}|u|^2dx+\frac{1}{r}\int\int_{Q_r(z)}|\nabla u|^2dxdt
\\&+\frac{1}{r^2}\int\int_{Q_r(z)}|u|^3dxdt+\frac{1}{r^2}\int\int_{Q_r(z)}|P|^\frac{3}{2}dxdt\notag\\\leq &
C(\rho,M),\notag
\end{align}where $Q_{2r}(z)\subset Q_\rho(z)\subset (0,T]\times\Omega$. In particular, fix $z\in(0,T]\times\Omega$, there exists a constant $r_0$ depending on $d_z=dist(z,\partial([0,T]\times\Omega))$ such that if $r<r_0$ then it is holding that
\begin{align}
&\sup_{t-r^2\leq s<t}\frac{1}{r}\int_{B_r(x)}|u|^2dx+\frac{1}{r}\int\int_{Q_r(z)}|\nabla u|^2dxdt
\\&+\frac{1}{r^2}\int\int_{Q_r(z)}|u|^3dxdt+\frac{1}{r^2}\int\int_{Q_r(z)}|P|^\frac{3}{2}dxdt\notag\\\leq &
C(M)\notag.
\end{align}
Moreover, we have
\begin{align}
d(x,\mathcal{E})\geq1;~~D\geq 1.
\end{align}
\end{theorem}
As an application of Theorem \ref{main1}, we first show that
\begin{theorem}\label{main2}
Assume $\Omega\subseteq\mathbb{R}^3$. Let $(u, P)$ be a solution to the  Navier-Stokes equations on $[0,T)\times\Omega$ which is regular on the time interval $[0,T)$. If $u$ experiences singularity in time at $t=T$ and $||u||_ {L^{2,\infty}([0,T];BMO(\Omega))}=M<\infty$ (or $\limsup_{t\rightarrow T^-}\sqrt{T-t}||u(t)||_{BMO(\Omega)}<\infty$), then it is holding
\begin{align*}
&\int_{\Omega}|u(t)|^2\phi dx+2\int_{0}^t\int_{\Omega}|\nabla u|^2\phi dxdt\\
=&\int_{0}^t\int_{\Omega}|u|^2(\partial_t\phi+\Delta\phi)dxdt+\int_{0}^t\int_{\Omega}(|u^2|+2P)u\cdot\nabla\phi
dxdt~for~all~t\in[0,T]\notag
\end{align*}where $\phi\in C^\infty_0((0,T]\times\Omega)$.
\end{theorem}

The second application of Theorem \ref{main1} is
\begin{theorem}\label{main3}
 Assume $(u,P)$ be a smooth solution of the Navier-Stokes equations \eqref{NS} on $[0,T)\times\mathbb{R}^3$ and $u\in L^{2,\infty}([0,T);BMO(\mathbb{R}^3))$ (or $\limsup_{t\rightarrow T^-}\sqrt{T-t}||u(t)||_{BMO(\mathbb{R}^3)}<\infty$) then the energy equality must hold on interval $[0,T]$
 \begin{align}
 \int_{\mathbb{R}^3}|u(t,\cdot)|^2dx+2\int_0^t\int_{\mathbb{R}^3}|\nabla u|^2dxdt=\int_{\mathbb{R}^3}|u_0|^2dx~for~t\in[0,T].
 \end{align}
\end{theorem}
Thirdly, we show a necessary condition for a solution $u$ of Navier-Stokes equations \eqref{NS} developing a singularity at time $T$
\begin{theorem}\label{main4}
Let $(u,P)$ be a solution to the Navier-Stokes equations \eqref{NS} on $[0,T)\times\mathbb{R}^3$ which is regular on the time interval $[0,T)$. If
$u$ satisfies $||u(t)||_{L^{2,\infty}([0,T];{BMO(\mathbb{R}^3)})}\leq c$ for suitable small constant $c$ then $u$ can be continued to the regular solution beyond $t=T$. Especially, if $u$ satisfies $$\limsup_{t\rightarrow T^-}\sqrt{T-t}||u(t)||_{BMO(\mathbb{R}^3)}\leq c,$$ then $u$ can be continued to the regular solution beyond $t=T$.
\end{theorem}
\begin{remark}
Theorem \ref{main4} means that as the solution $u$ of \eqref{NS} approaches a finite blow-up time $T$, the norm $||u(t)||_{BMO(\mathbb{R}^3)}$ must blow up at a rate $\frac{c}{\sqrt{T-t}}$.
\end{remark}
Furthermore, we can allow the vertical part of the velocity $u$ to be large.
\begin{theorem}\label{main5}
Let $(u,P)$ be a suitable weak solution to \eqref{NS} on $[-1,0]\times\mathbb{R}^3$. Assume that $u$ satisfies
\begin{align}
||u_3||_{L^{2,\infty}([-1,0];BMO(\mathbb{R}^3))}=M<\infty.
\end{align} Then there exists a suitable small constant $\varepsilon_1$ depending on $M$ such that if $||u_h||_{L^{2,\infty}([-1,0];BMO(\mathbb{R}^3))}\leq\varepsilon_1$ then $u$ is regular in $[-1,0]\times\mathbb{R}^3$.
\end{theorem}

By using Theorem \ref{main3}, we obtain some corollaries.
\begin{corollary}\label{corollary1}
Assume $(u,P)$ be a Leray-Hopf weak solution to the Navier-Stokes equations \eqref{NS} on $[0,T]\times\mathbb{R}^3$. If $u$ experiences a finite number of singularities at $\{t_i\}, 1\leq i\leq N$ in $[0,T]$ and $u$ satisfies $\limsup_{t\rightarrow t_i}\sqrt{t_i-t}||u(t)||_{BMO(\mathbb{R}^3)}<\infty$, then $u$ must satisfy energy equality on $[0,T]$.
\end{corollary}
\begin{proof}
Without loss of generality, we assume $t_0\in(0,T)$ and $u$ experiences singularity at $t_0$ and $T$. It is clear that $u$ is regular in $[0,t_0)$ and $(t_0,T)$. By Theorem \ref{main3}, we have that the energy equality must hold on $[0,t_0]$ and $[t_0+\epsilon,T]$ for $\epsilon\in (0,\frac{1}{2}(T-t_0))$
\begin{align}\label{eq1}
\int_{\mathbb{R}^3}|u(s)|^2dx+\int_{0}^s\int_{\mathbb{R}^3}|\nabla u|^2dxdt=&\int_{\mathbb{R}^3}|u_0|^2dx~for~ s\in[0,t_0],\\
\int_{\mathbb{R}^3}|u(\tau)|^2dx+\int_{t_0+\epsilon}^\tau\int_{\mathbb{R}^3}|\nabla u|^2dxdt=&\int_{\mathbb{R}^3}|u(t_0+\epsilon)|^2dx~for~ \tau\in[t_0+\epsilon,T].
\end{align}

To show the energy equality on $[0,T]$, it is sufficient to prove
\begin{align}\label{sc}
\lim_{\epsilon\rightarrow0^+}||u(t_0+\epsilon)||^2_{L^2(\mathbb{R}^3)}=||u(t_0)||^2_{L^2(\mathbb{R}^3)}.
\end{align}

On the one hand, by the weak lower semicontinuity of the $L^2$ norm of $u$, we know that
\begin{align}\label{wc}
||u(t_0)||^2_{L^2(\mathbb{R}^3)}\leq \liminf_{\epsilon\rightarrow0^+}||u(t_0+\epsilon)||^2_{L^2(\mathbb{R}^3)}.
\end{align}
On the other hand, it is well-known that the Leray-Hopf weak solutions on $\mathbb{R}^3$ satisfy the following strong energy inequality
\begin{align}\label{se}
||u(t)||^2_{L^2(\mathbb{R}^3)}+\int_s^t\int_{\mathbb{R}^3}|\nabla u(\tau,x)|^2dxd\tau\leq ||u(s)||^2_{L^2(\mathbb{R}^3)}~for~a.e.~s\in[0,t].
\end{align}
Noticing that \eqref{eq1} implies $\lim_{s\rightarrow t_0^-}||u(s)||^2_{L^2(\mathbb{R}^3)}=||u(t_0)||^2_{L^2(\mathbb{R}^3)}$. This means by \eqref{se}
\begin{align}\label{se1}
||u(t)||^2_{L^2(\mathbb{R}^3)}+\int_{t_0}^t\int_{\mathbb{R}^3}|\nabla u(\tau,x)|^2dxd\tau\leq ||u(t_0)||^2_{L^2(\mathbb{R}^3)}~for~ t\geq t_0.
\end{align} This implies that
\begin{align}\label{sc1}
||u(t_0)||^2_{L^2(\mathbb{R}^3)}\geq \limsup_{\epsilon\rightarrow0^+}||u(t_0+\epsilon)||^2_{L^2(\mathbb{R}^3)}.
\end{align}According to \eqref{sc1} and \eqref{wc}, we conclude that \eqref{sc} is holding.
\end{proof}
If $u$ is a suitable weak solution, by using Theorem \ref{main1} and a similar argument as \cite{Leslie}, we can obtain a stronger conclusion.
\begin{corollary}\label{corollary2}
Assume $(u,P)$ be a suitable weak solution to the Navier-Stokes equations \eqref{NS} on $[0,T]\times\mathbb{R}^3$ and $u\in L^{2,\infty}([0,T],BMO(\mathbb{R}^3))$, then $u$ must satisfy the energy equality on $[0,T]$.
\end{corollary}
\begin{proof}
From Theorem \ref{main1}, we obtain that $u\in L^\infty((0,T];\mathcal{M}^{2,1})$, where $\mathcal{M}^{2,1}$ is the Morrey space with integrability $2$ and rate $1$. Since the Morrey space $\mathcal{M}^{2,1}$ is invariant under shifts $f\to f(\cdot-x_0)$ and the rescaling $f(x)\to \lambda f(\lambda x)$, we have by Cannone's Theorem \cite{Cannone} that $u\in L^\infty B^{-1}_{\infty,\infty}$. Consequently, interpolation with the enstrophy space $L^2H^1=L^2B^1_{2,2}$ puts the solution $u$ into the Onsager-critical class $L^3((0,T]B^1_{3,3})$, from which we conclude that
\begin{align*}
\int_{\mathbb{R}^3}|u(t,\cdot)|^2dx+2\int_0^t\int_{\mathbb{R}^3}|\nabla u|^2dxdt=\int_{\mathbb{R}^3}|u(t_0,\cdot)|^2dx~for~a.e.~t_0\in(0,t].
\end{align*}By using the fact $\lim_{t_0\to0}||u(t_0)||_{L^2}=||u_0||_{L^2}$, we deduce the energy conservation
\begin{align*}
\int_{\mathbb{R}^3}|u(t,\cdot)|^2dx+2\int_0^t\int_{\mathbb{R}^3}|\nabla u|^2dxdt=\int_{\mathbb{R}^3}|u_0|^2dx~for~t\in[0,T].
\end{align*}
\end{proof}
\begin{remark}
Proposition \ref{prop} and Theorem \ref{main3} imply that
 though the solution $u$ to \eqref{NS}  experiences some potential type II singularity at first blow up time $T$, the energy equality may still holds on $[0,T]$. Hence,
Theorem \ref{main1}-\ref{main3} and Corollary \ref{corollary1}-\ref{corollary2} improve considerably the recent results in \cite{Leslie,Kozono}. Theorem \ref{main4}-\ref{main5} show a precise blow up rate for the $BMO$ norm of the solution $u$ to \eqref{NS} and improve considerably the results in \cite{Wang,He,Kozono1}.
\end{remark}
\begin{remark}
In recent work of Cheskidov-Luo \cite{Luo} showed that if  $u \in L^{2, \infty}\left([0, T] ; B_{\infty, \infty}^{0}(\mathbb{R}^3)\right)$, then the weak solution $u$ of \eqref{NS} equations satisfies energy equality. Where, we illustrate that $B M O \not \subset B_{\infty, \infty}^{0}$. For example, let $h(x)=\ln |x|$. Obviously, $h$ is in $BMO(\mathbb{R}^3)$. We shall show that $h$ is not in $B_{\infty, \infty}^{0}(\mathbb{R}^3)$. Indeed,
$$\left\|\ln|x|\right\|_{B^{0}_ {\infty, \infty}}=\sup _{j \geqslant-1}\left\| \Delta_{j} \ln |x|\right\|_{L^\infty},$$
where $$
	\Delta_{j} \ln |x| =\mathcal{F}^{-1}\left(\varphi (2^{-j} \xi) \widehat{\ln} |\xi|\right) =2^{j d} \int_{\mathbb {R}^{d}} \check{\varphi}\left(2^{j}(x-y)\right) \ln |y| d y.$$
Let $2^{j} y=z$, one has
$$
\begin{aligned}
\Delta_{j} \ln |x|&=\int_{\mathbb{R}^d} \check{\varphi}\left(2^{j} x-z\right) \ln \left(2^{-j} z\right) d z\\
&=\int_{\mathbb{R}^d} \check{\varphi}\left(2^{j} x-z\right)\left( \ln 2^{-j}+\ln z\right) d z\\
&=\int_{\mathbb{R}^d} \check{\varphi}\left(2^{j} x-z\right) \ln 2^{-j} dz +\int_{\mathbb{R}^d} \check{\varphi}\left(2^{j} x-z\right)\ln z d z.\\
\end{aligned}
$$
Thus we have
$$
\begin{aligned}
 	\Delta_{j} \ln (0)&=\int_{\mathbb{R}^d} \check{\varphi}\left(-z\right) \ln 2^{-j} dz +\int_{\mathbb{R}^d} \check{\varphi}\left(-z\right)\ln z d z\\
 	&=-j \ln 2\int_{\mathbb{R}^d} \check{\varphi}\left(-z\right)  dz +\int_{\mathbb{R}^d} \check{\varphi}\left(-z\right)\ln z d z.
\end{aligned}
 $$
Due to $\check{\varphi} \in\mathcal{S}(\mathbb{R}^d)$, one has that
$$
\int_{\mathbb{R}^d} \check{\varphi}\left(-z\right)\ln z d z< \infty.
$$
We now choose $ \int_{\mathbb{R}^d} \check{\varphi}\left(-z\right)  dz\neq 0$,
then
$$\lim_{j\to\infty}|j \ln 2\int_{\mathbb{R}^d} \check{\varphi}\left(-z\right)  dz|=\infty.$$
Clearly,
$$\left\| \Delta_{j} \ln |x|\right\|_{L^\infty} \sim j$$
and thus
$$\left\|\ln|x|\right\|_{B^{0}_ {\infty, \infty}}=\sup _{j \geqslant-1}\left\| \Delta_{j} \ln |x|\right\|_{L^\infty}=\infty.$$
It turns out that our result on energy conservation is of independent interest compare to the result of Cheskidov and Luo.
\end{remark}

Our paper is organized as follows: in Section 2, we establish some uniform bounds on invariant quantities $A(r)$, $B(r)$, $C(r)$ and $D(r)$; in Section 3, we prove Theorem \ref{main1}-\ref{main5} and Proposition \ref{prop}.

\section{The uniform bounds on invariant quantities}
\par
~~~~Our goals in this section are to show some uniform bounds on invariant quantities under the assumption $u\in L^{2,\infty}([0,T];BMO(\Omega))$.
\begin{lemma}\label{CB}
Let $(u,P)$ be a suitable weak solution of \eqref{NS} in $[0,T]\times\Omega$ satisfying
\begin{align*}
||u||_{L^{2,\infty}([0,T];BMO(\Omega))}=M<\infty.
\end{align*}Then it is holding
\begin{align*}
C(u,r,z)\leq C\frac{r}{\rho}C(u,\rho,z)+C(\frac{\rho}{r})^\frac{3}{2}M^\frac{3}{2}A^\frac{3}{4}(u,\rho,z)
\end{align*}for $0<r<\rho$ and $Q_\rho(z)\subset[0,T]\times\Omega$.
\end{lemma}
\begin{proof}
At almost every time we estimate
\begin{align}
\int_{B_r(x)}|u|^3dx\leq\int_{B_r(x)}|u-u_{\rho}|^3dx+C|B_r||u_{\rho}(x)|^3
=I_1+I_2
\end{align}where $u_\rho=\frac{1}{|B_{\rho}|}\int_{B_\rho(x)}udy$.

For $I_1$, we have
\begin{align}
I_1=&\int_{B_r(x)}|u-u_\rho|^\frac{3}{2}|u-u_\rho|^\frac{3}{2}dx\\\leq&(\int_{B_\rho(x)}|u-u_\rho|^2)^\frac{3}{4}
(\frac{1}{|B_\rho|}\int_{B_\rho(x)}|u-u_\rho|^6)^\frac{1}{4}|B_\rho|^{\frac{1}{4}}\notag\\
\leq&\rho^\frac{3}{2}A(\rho,z)^\frac{3}{4}||u(t)||^\frac{3}{2}_{BMO(\mathbb{R}^3)}.\notag
\end{align}

For $I_2$, we have
\begin{align}
I_2=|B_r(x)||\frac{1}{|B_\rho(x)|}\int_{B_\rho(x)}udy|^3\\
\leq C(\frac{r}{\rho})^3\int_{B_\rho(x)}|u|^3dx.\notag
\end{align}

Summing up the estimates for $I_1$ and $I_2$ and integrating with respect to time from $t-r^2$ to $t$, we obtain
\begin{align}\label{2.4}
C(u,r,z)=\frac{1}{r^2}\int\int_{Q_r(z)}|u|^3dxdt\leq
C\frac{r}{\rho}C(u,\rho,z)+(\frac{\rho}{r})^\frac{3}{2}r^\frac{-1}{2}\int_{t-r^2}^t||u(s)||^\frac{3}{2}_{BMO}ds
A^\frac{3}{4}(u,\rho,z).
\end{align}

 Using the assumption $||u||_{L^{2,\infty}([0,T];BMO(\mathbb{R}^3))}=M$, we deduce
\begin{align}\label{2.5}
&r^{-\frac{1}{2}}\int_{t-r^2}^t||u||^\frac{3}{2}_{BMO(\mathbb{R}^3)}ds\\
=&\frac{3}{2}r^{-\frac{1}{2}}\int_{0}^{\infty}\sigma^{\frac{1}{2}}|\{s\in[t-r^2,t]:||u(s,\cdot)||_{BMO(\mathbb{R}^3)}
>\sigma\}|d\sigma\notag\\
=&\frac{3}{2}r^{-\frac{1}{2}}\{\int_{0}^{R}\sigma^{\frac{1}{2}}|\{s\in[t-r^2,t]:||u(s,\cdot)||_{BMO(\mathbb{R}^3)}
>\sigma\}|d\sigma\notag\\&+
\int_{R}^{\infty}\sigma^{\frac{1}{2}}|\{s\in[t-r^2,t]:||u(s,\cdot)||_{BMO(\mathbb{R}^3)}>\sigma\}|d\sigma\}\notag\\
\leq&r^{-\frac{1}{2}}(R^\frac{3}{2}r^2+3R^{-\frac{1}{2}}M^2)\notag\\
=&2(3^{\frac{3}{4}})M^\frac{3}{2}\notag
\end{align}where we choose $R=\sqrt{3}r^{-1}M$. Substituting \eqref{2.5} into \eqref{2.4}, we complete the proof of this lemma.
\end{proof}
By using Lemma \ref{CB}, we deduce some uniform bounds on invariant quantities:
\begin{lemma}\label{ubi}
Let $(u,P)$ be a suitable weak solution of \eqref{NS} in $[0,T]\times\Omega$ satisfying
\begin{align*}
||u||_{L^{2,\infty}([0,T];BMO(\Omega))}=M<\infty.
\end{align*}Then it is holding
\begin{align}\label{estimate}
A(u,r,z)+ B(u,r,z)+C^\frac{7}{6}(u,r,z)+D^\frac{8}{7}(P,r,z)\leq C(d_z,M)~for~Q_{2r}(z)\subset Q_{d_z}(z)\subset
(0,T]\times\Omega
\end{align}where $d_z=dist(z,\partial((0,T]\times\Omega))$. In particular, fix $z\in(0,T]\times\Omega$, there exists a constant $r_0>0$ depending on $d_z$, such that if $r<r_0$ and $Q_r(z)\subset(0,T]\times\mathbb{R}^3$ then it is holding
\begin{align}\label{estimatew}
A(u,r,z)+B(u,r,z)+C^\frac{7}{6}(u,r,z)+D^\frac{8}{7}(P,r,z)\leq  C(M).
\end{align}
\end{lemma}
\begin{proof}

 Without loss of generality, we set $z=0$ and $\rho\leq d_z$.
Let $\phi(t,x)=\chi(t,x)\psi(t,x)$ where $\chi$ is cut-off function which equals 1 in
$Q_{\frac{1}{2}\rho}$ and vanishes outside of $Q_{\frac{3}{4}\rho}$. Then let $\psi=(4\pi(r^2-t))^{-\frac{3}{2}}e^{-\frac{|x|^2}{4(r^2-t)}}$. Direct computations show that $\phi\geq0$ and
\begin{align*}
\partial_t\phi+\triangle\phi=&0~in~Q_{\frac{1}{2}\rho},\\
|\partial_t\phi+\triangle\phi|\leq& C\rho^{-5}~in~Q_{\rho},\\
C^{-1}r^{-3}\leq\phi\leq Cr^{-3};~|\nabla\phi|\leq& Cr^{-4}~in~Q_r,\\
\phi\leq C\rho^{-3};~|\nabla\phi|\leq& C\rho^{-4}~in~Q_\rho-Q_{\frac{3}{4}\rho}.
\end{align*}
Using $\phi$ as a test function in the local energy inequality \eqref{le}, we obtain
\begin{align}\label{AB}
A(u,r)+B(u,r)\leq& C(\frac{r}{\rho})^2A(u,\rho)+C(\frac{\rho}{r})^2C(u,\rho)+C(\frac{\rho}{r})^2C^\frac{1}{3}(u,\rho)D^\frac{2}{3}(P,\rho)\\
\leq& C(\frac{r}{\rho})^2A(u,\rho)+C(\frac{\rho}{r})^2C(u,\rho)+C(\frac{\rho}{r})^2D(P,\rho)\notag.
\end{align}

For $C(u,\rho)$, we have by using Lemma \ref{CB}
\begin{align}\label{C}
C(u,r)\leq C\frac{r}{\rho}C(u,\rho)+C(\frac{\rho}{r})^\frac{3}{2}M^\frac{3}{2}A^\frac{3}{4}(u,\rho).
\end{align}

It remains to show some bounds on $D(u,r)$. Let $\eta(x)$ be a cut-off function which equals 1 in $B_{\frac{3\rho}{4}}$ and vanishes outside of $B_\rho$. Let $P_1$ satisfy $-\Delta P_1=\partial_{x_i}\partial_{x_j}(u_iu_j\eta)$ and $P_2=P-P_1$. Then, it is clear that $\Delta P_2=0$ in $B_\frac{3\rho}{4}$. by using the Calder\'{o}n-Zygmund inequality, we have
\begin{align*}
\int_{B_\rho}|P_1|^{\frac{3}{2}}dx\leq C(\int_{B_\rho}|u|^3dx).
\end{align*}By the properties of harmonic functions, we infer that for $r\leq\frac{\rho}{2}$,
\begin{align*}
\int_{B_r}|P_2|^{\frac{3}{2}}dx\leq Cr^3\sup_{x\in B_r}|P_2(x)|^\frac{3}{2}\leq C(\frac{r}{\rho})^3\int_{B_\rho}|P_2|^{\frac{3}{2}}dx.
\end{align*}It then follows that for $0<r\leq\frac{\rho}{2}$
\begin{align*}
&\int_{B_r}|P|^\frac{3}{2}dx\\\leq& C(\int_{B_\rho}|u|^3dx)+C(\frac{r}{\rho})^3\int_{B_\rho}|P-P_1|^\frac{3}{2}dx\\
\leq& C(\int_{B_\rho}|u|^3dx)+C(\frac{r}{\rho})^3\int_{B_\rho}|P|^\frac{3}{2}dx.
\end{align*}Integrating with respect to t from $-r^2$ to 0, we obtain, using H\"{o}lder inequality,
\begin{align*}
\int_{Q_r}|P|^{\frac{3}{2}}dxdt\leq C\int_{Q_\rho}|u^3|dxdt+C(\frac{r}{\rho})^3\int_{Q_\rho}|P|^\frac{3}{2}dxdt.
\end{align*} This implies
\begin{align}\label{D}
D(P,r)\leq C\frac{r}{\rho}D(P,\rho)+C(\frac{\rho}{r})^2C(u,\rho).
\end{align}

According to \eqref{D}, we have by using Young's inequality
\begin{align}\label{D1}
D(P,r)^\frac{8}{7}\leq& C(\frac{r}{\rho})^\frac{8}{7}D^\frac{8}{7}{(P,\rho)}+C(\frac{\rho}{r})^\frac{16}{7}C(u,\rho)^\frac{8}{7}\\
\leq& C(\frac{r}{\rho})^\frac{8}{7}D^\frac{8}{7}{(P,\rho)}+(\frac{r}{\rho})^\frac{7}{6}C(u,\rho)^\frac{7}{6}
+C(\frac{\rho}{r})^{168}\notag.
\end{align}

Similar computations show by using \eqref{C}
\begin{align}\label{C1}
C^\frac{7}{6}(u,r)\leq C(\frac{r}{\rho})^\frac{7}{6}C^\frac{7}{6}(u,\rho)+(\frac{r}{\rho})^2A(u,\rho)+C(\frac{\rho}{r})^{21}M^{14}.
\end{align}

By using Young's inequality, we deduce from \eqref{AB}
\begin{align}\label{AB1}
A(u,r)+B(u,r)\leq C(\frac{r}{\rho})^2A(u,\rho)+(\frac{r}{\rho})^\frac{7}{6}C^\frac{7}{6}(u,\rho)
+(\frac{r}{\rho})^\frac{8}{7}D^\frac{8}{7}(P,\rho)+C((\frac{\rho}{r})^{21}+(\frac{\rho}{r})^{24}).
\end{align}

Summing up the estimates \eqref{D1}, \eqref{C1} and \eqref{AB1}, we obtain by defining $G(r)= A(u,r)+B(u,r)+C^\frac{7}{6}(u,r)+D^\frac{8}{7}(P,r)$
\begin{align}\label{decay}
G(r)\leq C(\frac{r}{\rho})^\frac{8}{7}G(\rho)+(C+M^{14})(\frac{\rho}{r})^{168}
\end{align} where we have used the fact $\frac{r}{\rho}<1$.

 Fix $\theta=\min\{\frac{1}{2},\frac{1}{C^7}\}$ and set $r=\theta^{k}\rho$ for $k\in\mathbb{N}$, \eqref{decay} yields
\begin{align}\label{decay1}
G(\theta^{k}\rho)\leq\theta G(\theta^{k-1}\rho)+C(1+M^{14})\theta^{-168}.
\end{align}By a standard iterative argument, we deduce that
\begin{align}\label{result1}
G(r)\leq C\frac{r}{\rho}G(\rho)+C(1+M^{14})~for~ r\leq\frac{\rho}{2}.
\end{align}This means \eqref{estimate}.

In particular, by choosing $r_0$ satisfying $C\frac{r_0}{d_z}G(d_z)\leq M$, we obtain \eqref{estimatew}.

%{\bf In the case $\nu=0$}:

%In \eqref{eneq}, we choose $\phi$ satisfying $\phi=1$ in $Q_{\frac{\rho}{2}}(z)$ and $\phi=0$ outside of $Q_{\frac{3\rho}{4}}(z)$. Then we obtain by direct computations
%\begin{align*}
%|\nabla \phi|\leq\frac{C}{\rho},~|\partial_t u+\nu\Delta u|\leq\frac{C}{\rho^2}
%\end{align*} and
%\begin{align}\label{A2}
%&A(u,r,z)\\\leq&\frac{\rho}{r}C^\frac{2}{3}(u,\rho,z)+\frac{\rho}{r}C(u,\rho,z)+\frac{\rho}{r}D(P,\rho,z)\notag\\
%\leq& 2(\frac{r}{\rho})^\frac{7}{6}C^\frac{7}{6}(u,\rho,z)+(\frac{r}{\rho})^\frac{8}{7}D^\frac{8}{7}(P,\rho,z)
%+3(\frac{\rho}{r})^{16}\notag.
%\end{align} From \eqref{C1} and \eqref{D1}, we also deduce
%\begin{align}\label{D2}
%&D(P,r,z)^\frac{8}{7}\\\leq& C(\frac{r}{\rho})^\frac{8}{7}D^\frac{8}{7}{(P,\rho,z)}+C(\frac{\rho}{r})^\frac{16}{7}C^\frac{8}{7}(u,\rho,z)\notag\\
%\leq& C(\frac{r}{\rho})^\frac{8}{7}D^\frac{8}{7}{(P,\rho,z)}+(\frac{r}{\rho})^\frac{7}{6}C^\frac{7}{6}(u,\rho,z)
%+C(\frac{\rho}{r})^{168}\notag.
%\end{align}
%\begin{align}\label{C2}
%&C^\frac{7}{6}(u,r,z)\leq(\frac{r}{\rho})^\frac{7}{6}C^\frac{7}{6}(u,\rho,z)+\frac{r}{\rho}A(u,\rho,z)+C(\frac{\rho}{r})^{21}M^{14}.
%\end{align} We now define $F(r,z)\equiv A(u,r,z)+C^\frac{7}{6}(u,r,z)+D^\frac{8}{7}(P,r,z)$. A repeated argument as the case $\nu=1$ shows that
%\begin{align}
%F(r,z)\leq C\frac{r}{\rho}F(\rho,z)+C(M).
%\end{align}This implies \eqref{estimate}. In particular, by choosing $r_0$ satisfying $C\frac{r_0}{d_z}F(d_z)\leq M$, we obtain \eqref{estimatew}.
\end{proof}

We conclude this section by giving a uniform bound on $C(u,r,z)+D(P,r,z)$ when $||u||_{L^{2,\infty}(BMO(\mathbb{R}^3))}$ is suitable small.
\begin{lemma}\label{epsilon}
Let $(u,P)$ be a suitable weak solution of \eqref{NS} in $[0,T]\times\mathbb{R}^3$. Fix $\varepsilon>0$, there exists constant $c$ and $r^*$ such that if
\begin{align*}
||u||_{L^{2,\infty}([0,T];BMO(\mathbb{R}^3))}\leq c,
\end{align*}then it is holding
\begin{align}\label{estimate1}
C(u,r,z)+D(P,r,z)\leq\varepsilon, ~for~0<r\leq r^*,
\end{align}where $r^*$ is depending on $d_z$.
\end{lemma}

\begin{proof}
Without loss of generality, we assume $c\leq 1$. Firstly, from Lemma \ref{ubi}, we have that there exists $r_0>0$ depending on $d_z$ such that if $\rho\leq r_0$, then it is holding
\begin{align}\label{md}
&C(u,\rho,z)+D(P,\rho,z)+A(u,\rho,z)+B(u,\rho,z)\leq C.
\end{align} Then we deduce by choosing $\rho\leq r_0$ in \eqref{C} and using \eqref{md}
\begin{align*}
&C(u,r,z)\\\leq& \frac{r}{\rho}C(u,\rho,z)+C(\frac{\rho}{r})^{\frac{3}{2}}c^\frac{3}{2}A^\frac{3}{4}(u,\rho,z)\notag\\
\leq& \frac{r}{\rho}C(u,\rho,z)+Cc^\frac{3}{2}(\frac{\rho}{r})^\frac{3}{2}\notag
\end{align*}for $r\leq\rho$.
We now take $r=2^{-k}r_0$ and obtain by a standard iterative argument that
\begin{align}\label{mC}
&C(u,2^{-k}r_0,z)\\\leq& \frac{1}{2}C(u,2^{-k+1}r_0,z)+C(2c)^\frac{3}{2}\notag\\
\leq&2^{-k}C(u,r_0,z)+\sum_{j=0}^{k}2^{-j}C(2c)^{\frac{3}{2}}\notag\\
\leq&2^{-k}C(u,r_0,z)+C(2c)^{\frac{3}{2}}.\notag
\end{align}
Similarly, by using \eqref{D} and \eqref{mC}, we get
\begin{align}\label{mD}
&D(u,2^{-k_0k}r_0,z)\\\leq& \frac{1}{2}D(u,2^{(-k_0(k-1))}r_0,z)+C2^{2k_0}(2^{-k_0(k-1)}C(u,r_0,z)+C(2c)^\frac{3}{2})\notag\\
\leq&2^{-k}D(u,2^{-k_0}r_0,z)+C2^{2k_0}(2^{-k_0(k-1)}C(u,r_0,z)+C(2c)^\frac{3}{2})\notag
\end{align}where we choose $k_0$ satisfying $C2^{-k_0}\leq\frac{1}{2}$.

Summing up \eqref{mC} and \eqref{mD}, we obtain by using \eqref{md} that
\begin{align*}
&C(u,2^{-kk_0}r_0,z)+D(P,2^{-kk_0}r_0,z)\\\leq&(2^{-k}+2^{-kk_0})(C(u,r_0,z)+D(P,2^{-k_0}r_0,z))+C2^{2k_0}(2^{-kk_0+k_0}C(u,r_0,z)+(2c)^\frac{3}{2})\\\leq&
C(2^{-k}+2^{-kk_0})+C2^{2k_0}(2^{-kk_0+k_0}C+(2c)^\frac{3}{2}).
\end{align*} We now choose $K$ satisfying $C(2^{-K}+2^{-Kk_0}+2^{-Kk_0+3k_0})\leq\frac{\varepsilon}{2}$ and $r^*=2^{-Kk_0}r_0$, then we take $c=\frac{1}{2}(\frac{\varepsilon}{2^{2k_0+1}C})^{\frac{2}{3}}$ and conclude the Lemma \ref{epsilon}.
\end{proof}
\section{The proof of main theorems}
\par
~~~~We first recall two well-known results about the suitable weak solutions to \eqref{NS}.
\begin{proposition}\cite{Lin}\label{p1}
Let $(u^k,P^k)$ be a sequence of weak solutions of \eqref{NS} in $Q_1$
such that $$||u^k||_{L^\infty([-1,0];L^2(B))}+||P^k||_{L^\frac{3}{2}(Q)}+||\nabla u^k||_{L^2(Q)}\leq E$$ and $(u^k,P^k)$ satisfies the local energy inequality \eqref{le}. Suppose that $(u,P)$ is the weak limit of $(u^k,P^k)$; then (u,P) is a suitable weak solution of \eqref{NS} on $Q$.
\end{proposition}
\begin{proposition}\cite{Caffarelli,Lin,Ladyzhenskaya}\label{p2}
Let $(u,P)$ be a suitable weak solution of \eqref{NS} in $Q_r(z)$. There exists $\varepsilon_0>0$ such that if
\begin{align*}
\frac{1}{r^2}\int_{Q_r(z)}(|u|^3+|P|^\frac{3}{2})dxdt\leq\varepsilon_0,
\end{align*}then $u$ is regular in $Q_\frac{r}{2}(z)$.
\end{proposition}

{\bf We now begin to prove Theorem \ref{main1}}.
\begin{proof}
The estimates (1.8)-(1.9) are direct conclusions of lemma \ref{ubi}.

For any $z=(x,T)\in \Omega\times\{T\}$, set $d_x=dist(x,\partial\Omega)$. From (1.8) and the weak lower semicontinuity of the $L^2$ norm of $u$, we get
\begin{align*}
\mathcal{E}(B_r(x))\leq\lim_{t\rightarrow T^-}\int_{B_r(x)}|u(t)|^2dy\leq C(d_x,M)r~for~r\leq\frac{d_x}{2}.
\end{align*}By the definitions of local dimension $d(x,\mathcal{E})$, it is clear that $d(x,\mathcal{E})\geq1$.

We now prove $D\geq1$. For any compact $S\subset\Omega$ with $dim_{\mathcal{H}}(S)=s^*<1$, by the definition of Hausdorff dimension, we have $H^{s_1}(S)=0$ for $s^*<s_1<1$. Then fixing $0<\delta<dist(S,\partial\Omega)$, there exists a collection of open balls $\{B_{r_i}(x_i)\}_{i=1}^{\infty}$ satisfying $r_i\leq\delta$ such that $S\subset\cup_{i=1}^\infty B_{r_i}(x_i)\subset\Omega$ and $\sum_{i=1}^\infty r^{s_1}_i\leq1$. From this facts, we deduce
\begin{align*}
\mathcal{E}(S)\leq\sum_{i=1}^\infty\mathcal{E}(B_{r_i}(x_i))\leq C(M)\sum_{i=1}^\infty r_i
\leq C(M)\delta^{1-s_1}\sum_{i=1}^\infty r^{s_1}_i\rightarrow 0~as~\delta\rightarrow0.
\end{align*}This means $D\geq 1$.
\end{proof}

{\bf The proof of Theorem \ref{main2}% is motivated by Seregin \cite{Seregin1} and based on Lemma \ref{ubi}}.
\begin{proof}
Assume $supp\phi(T,x)\equiv\bar{\Omega}_0\subset\Omega$ and $d=dist(supp\phi,\partial(\Omega\times[0,T]))$. From the proof of Lemma \ref{ubi}, it is clear that
\begin{align}\label{UA}
\sup_{T-r^2\leq t<T,x\in\bar{\Omega}_0}\frac{1}{r}\int_{B_r(x)}|u(t)|^2dy\leq C(M,d)~for~0<r\leq\frac{d}{2}.
\end{align}

By the weak lower semicontinuity of the $L^2$ norm of $u$, we see that to prove Theorem \ref{main2} it is enough to prove
\begin{align}\label{aim}
\lim_{t\rightarrow T^-}\int_{\bar{\Omega}_0}|u(t)|^2dx=\int_{\bar{\Omega}_0}|u(T)|^2dx.
\end{align}

Let $\Sigma$ be the set of all singular points of $u$ at time $T$. By the well-known Caffarelli-Kohn-Nirenberg Theorem \cite{Caffarelli}, we have
\begin{align*}
H^1(\Sigma)=0.
\end{align*}This fact implies that for each $\epsilon>0$, there exists a countable family of sets of the form
\begin{align*}
B^\epsilon_i=B_{r^\epsilon_i}(x^\epsilon_i)\times\{t=T\}
\end{align*}such that
\begin{align}\label{HM}
r^\epsilon_i\leq\frac{d}{2},~\Sigma\cap(\bar{\Omega}_0\times\{t=T\})\subset\cup B^\epsilon_i,~\sum r^\epsilon_i<\epsilon.
\end{align}

Fix $\varepsilon>0$ and let $$\epsilon=\frac{\varepsilon}{8C(||u_0||_{L^2},M,d)}.$$
Then, by \eqref{UA} and \eqref{HM}, we have
\begin{align}\label{sb}
&|\int_{\cup B^\epsilon_i}|u(t)|^2dx-\int_{\cup B^\epsilon_i}|u(T)|^2dx|\\
\leq&\sum_i\int_{B^\epsilon_i}|u(t)|^2dx+\sum_i\int_{B^\epsilon_i}|u(T)|^2dx\notag\\
\leq& 2C(||u_0||_{L^2},M,d)\sum_ir^\epsilon_i\notag\\<&2C(||u_0||_{L^2},M,d)\epsilon\notag\\
\leq&\frac{\varepsilon}{4}\notag
\end{align}for all $t\in[T-(\frac{d}{2})^2,T]$, where we have used the fact $\int_{\cup B^\epsilon_i}|u(T)|^2dx\leq\liminf_{t\rightarrow T^-}\int_{\cup B^\epsilon_i}|u(t)|^2dx$ due to the weak lower semicontinuity of the $L^2$ norm of $u$.

Set $$\omega^\epsilon\equiv(\bar{\Omega}_0\times\{t=T\})-\cup B^\epsilon_i.$$

For each $z\in\omega^\epsilon$, there exists a non-empty neighborhood $\mathcal{O}_z$ such that $u$ is H\"{o}lder continuous on $\mathcal{O}_z\cap([0,T]\times\bar{\Omega}_0)$. Since $\omega^\epsilon$ is compact, there
exists a non-empty neighborhood $\mathcal{O}^\epsilon_\omega$ of the set $\omega^\epsilon$ such that $u$ is continuous in $\bar{\mathcal{O}^\epsilon_\omega}\cap([0,T]\times\bar{\Omega}_0).$ Hence,
\begin{align}\label{rb}
|\int_{\omega^\epsilon}|u(t)|^2dx-\int_{\omega^\epsilon}|u(T)|^2dx|<\frac{\varepsilon}{2}
\end{align} for $|T-t|$ small enough. Combining \eqref{sb} and \eqref{rb}, we obtain
\begin{align}
&\lim_{t\rightarrow T^-}|\int_{\Omega_0}|u(t)|^2dx-\int_{\Omega_0}|u(T)|^2dx|\\
\leq&\limsup_{t\rightarrow T^-}|\int_{\omega^\epsilon}|u(t)|^2dx-\int_{\omega^\epsilon}|u(T)|^2dx|+\limsup_{t\rightarrow T^-}|\int_{\cup B^\epsilon_i}|u(t)|^2dx-\int_{\cup B^\epsilon_i}|u(T)|^2dx|\notag\\
<&\varepsilon.\notag
\end{align}This implies \eqref{aim} by taking $\varepsilon\rightarrow0$.
\end{proof}

{\bf We now begin to prove Theorem \ref{main3}}
\begin{proof}
Since $u\in L^{2,\infty}([0,T),BMO(\mathbb{R}^3))$, we deduce that from Theorem \ref{main2}
\begin{align}
&\int_{\mathbb{R}^3}|u(t)|^2\phi dx+2\int_{0}^t\int_{\mathbb{R}^3}|\nabla u|^2\phi dxdt\\
=&\int_{\mathbb{R}^3}|u_0|^2\phi dx+\int_{0}^t\int_{\mathbb{R}^3}|u|^2(\partial_t\phi+\triangle\phi)dxdt+\int_{0}^t\int_{\mathbb{R}^3}(|u^2|+2P)u\cdot\nabla\phi
dxdt~for~all~t\in[0,T]\notag
\end{align}where $\phi\in C^\infty_0([0,T]\times\mathbb{R}^3)$. Let $\phi(t,x)=\eta(x)$ be a cut-off function which equals $1$ in $B_R$ and vanishes outside of $B_{2R}$ for $R>0$. Then we obtain that
\begin{align}
&\int_{\mathbb{R}^3}|u(t)|^2\eta dx+2\int_{0}^t\int_{\mathbb{R}^3}|\nabla u|^2\eta dxdt\\
=&\int_{\mathbb{R}^3}|u_0|^2\eta dx+\int_{0}^t\int_{\mathbb{R}^3}|u|^2\triangle\eta dxdt+\int_{0}^t\int_{\mathbb{R}^3}(|u^2|+2P)u\cdot\nabla\eta
dxdt~for~all~t\in[0,T]\notag
\end{align}

Noticing that
\begin{align}
|\int_{0}^t\int_{\mathbb{R}^3}|u|^2\triangle\eta dxdt|&\leq C|T|\sup_{t\in[0,T)}||u(t)||^2_{L^2}\frac{1}{R^2},\\
|\int_{0}^t\int_{\mathbb{R}^3}(|u^2|+2P)u\cdot\nabla\eta dxdt|&\leq C(||u||^3_{L^3([0,T];L^3(\mathbb{R}^3))}+||P||^\frac{3}{2}_{L^\frac{3}{2}([0,T];L^\frac{3}{2}(\mathbb{R}^3))})
\frac{1}{R}.
\end{align} Taking $R\rightarrow\infty$, we deduce that from (3.8)
\begin{align}
\int_{\mathbb{R}^3}|u(t)|^2dx+2\int_{0}^t\int_{\mathbb{R}^3}|\nabla u|^2dxdt=\int_{\mathbb{R}^3}|u_0|^2 dx.
\end{align}
\end{proof}

{\bf We now prove Theorem \ref{main4}}:
\begin{proof}
Fix $z\in(0,T]\times\mathbb{R}^3$. From Lemma \ref{epsilon}, there exists $c>0$ and $0<r^*<d_z$ such that if $||u||_{L^{2,\infty}([0,T];BMO(\mathbb{R}^3))}\leq c$ then it holds
\begin{align}
C(u,r,z)+D(P,r,z)<\varepsilon_0~for~0<r<r^*.
\end{align}By Proposition \ref{p2}, $u$ is regular on $[0,T]\times\mathbb{R}^3$. This complete the proof of Theorem \ref{main4}.
\end{proof}

{\bf We now turn to the proof of Theorem \ref{main5}}, this proof is based on the uniform bounds on invariant quantities established in Lemma \ref{ubi} and a standard blow-up argument.
\begin{proof}
If the conclusion of the theorem is false, then there exists a constant $M$ and a sequence $(u^k,P^k)$ whose elements are suitable weak solutions of \eqref{NS} in $[-1,0]\times\mathbb{R}^3$ and satisfy
\begin{align*}
||u^k_3||_{L^{2,\infty}([-1,0];BMO(\mathbb{R}^3))}\leq M,~||u^k_h||_{L^{2,\infty}([-1,0];BMO(\mathbb{R}^3))}\leq\frac{1}{k}.
\end{align*}Furthermore, we can assume $u^k$ is singular at $(0,0)$. According to Lemma \ref{ubi}, there exists a sequence of $\{r_k\}$ satisfying
\begin{align*}
C(u^k,r)+D(P^k,r)\leq C(M)~for~0<r\leq r_k.
\end{align*}
Set $v^k=r_ku(r^2_kt,r_kx)$ and $\Pi^k=r^2_kP(r^2_kt,r_kx)$, then it is clear that $(v^k,\Pi^k)$ are suitable weak solutions to \eqref{NS} on $Q$ and satisfy
\begin{align*}
||v^k_3||_{L^{2,\infty}([-1,0];BMO(B))}\leq M,&~||v^k_h||_{L^{2,\infty}([-1,0];BMO(B))}\leq\frac{1}{k};\\
C(v^k,1)+&D(\Pi^k,1)\leq C(M).
\end{align*}
By the local energy inequality \eqref{le}, we obtain
\begin{align*}
||\partial_t v^k||_{L^\frac{3}{2}}([-1,0];H^{-2}(B_{\frac{3}{4}}))+||v^k||_{L^\infty([-1,0];L^2(B_{\frac{3}{4}}))}+||\nabla v^k||_{L^2([-1,0];L^2(B_{\frac{3}{4}}))}
\leq C(M).
\end{align*} According to the well-known Aubin-Lions lemma and Proposition \ref{p1}, by passing to a subsequence if necessary, we may deduce that there exists a suitable weak solution $(v,\Pi)$ to \eqref{NS} such that
\begin{align*}
v^k\rightarrow v~in~L^3(Q_{\frac{3}{4}}),&~v^k\rightharpoonup v~\in~L^\infty([-\frac{9}{16},0];L^2(B_\frac{3}{4})),\\~\Pi^k\rightharpoonup \Pi~in~L^\frac{3}{2}(Q_{\frac{3}{4}}),&~\nabla v^k\rightharpoonup\nabla v~in~L^2(Q_{\frac{3}{4}})
\end{align*}as $k\to\infty$. In particular, we have by Fatou's lemma
\begin{align*}
||v_h||_{L^{2,\infty}([-1,0];BMO(B_\frac{3}{4}))}\leq\liminf_{k\to\infty}
||v^k_h||_{L^{2,\infty}([-1,0];BMO(B_\frac{3}{4}))}=0.
\end{align*} Therefore, we get $v_h=b(t)$ with $b(t)\in L^\infty[-\frac{9}{16},0]$. Then by $\nabla\cdot v=0$, we obtain $\partial_3 v_3=0$. Thus, we deduce that $(v,\Pi)$ satisfies
\begin{equation}
\left\{\begin{array}{ll}
\partial_t v-\Delta v+\nabla \Pi=-(b_1(t),b_2(t),0)\cdot\nabla(0,0,v_3)~in~Q_{\frac{3}{4}},\\
\nabla\cdot v=0.
\end{array}\right.
\end{equation}By the classical result concerning linear Stokes equations \cite{Ladyzhenskaya1} and the fact $b(t)\in L^\infty[-\frac{9}{16},0]$, we have that $|v|\leq C(M)$ in $Q_\frac{1}{2}$. However, $(0,0)$ is a singular point of $v^k$, hence, it is holding by using Proposition \ref{p2}, \eqref{D} and the fact $v^k\to v~in~L^3(Q_{\frac{3}{4}})$
\begin{align}
\varepsilon_0&<\liminf_{k\to\infty}\frac{1}{r^2}\int_{Q_r}(|v^k|^3+|\Pi^k|^\frac{3}{2})dxdt\\
&\leq \liminf_{k\to\infty}(C(v,r)+\frac{r}{\rho}D(\Pi^k,\rho)+\frac{\rho^2}{r^2}C(v^k,\rho))\notag\\
&=C(v,r)+\frac{\rho^2}{r^2}C(v,\rho)+\liminf_{k\to\infty}\frac{r}{\rho}D(\Pi^k,\rho)\notag
\end{align}for $0<2r<\rho<\frac{1}{2}$. From Lemma \ref{ubi} and the fact $|v|\leq C(M)$ in $Q_\frac{1}{2}$, we get $D(\Pi^k,\rho)\leq C(M)$ and
\begin{align*}
\varepsilon_0<C(M)(r^3+\frac{\rho^5}{r^2}+\frac{r}{\rho}).
\end{align*} By taking $\rho=\sqrt{r}$, we infer
\begin{align*}
\varepsilon_0<C(M)r^\frac{1}{2}.
\end{align*} which, for sufficiently small $r$, is a contradiction.
\end{proof}

{\bf Finally, we prove Proposition \ref{prop}}.
\begin{proof}
Firstly, a direct computation shows that
\begin{align}
||u(t)||^2_{L^2}\sim\sqrt{T-t},\quad ||\nabla u(t)||^2_{L^2}\sim\frac{1}{\sqrt{T-t}}.
\end{align}

Secondly, we observe that $||u(\lambda x)||_{L^\infty}=||u(x)||_{L^\infty}$ and $||u(\lambda x)||_{BMO}=||u(x)||_{BMO}$. It is clear that
\begin{align*}
||u(t)||_{L^\infty}=&\frac{1}{\sqrt{T-t}}||\mathbb{P}(|x|^{\sqrt{T-t}}(-\ln|x|)^s \phi(x))||_{L^\infty},\\
||u(t)||_{BMO}=&\frac{1}{\sqrt{T-t}}||\mathbb{P}(|x|^{\sqrt{T-t}}(-\ln|x|)^s \phi(x))||_{BMO}.
\end{align*}
Let $v=\frac{1}{\sqrt{T-t}}|x|^{\sqrt{T-t}}(-\ln|x|)^{s}\phi(x)$, it is clear that $||u||_{L^\infty}=||\mathbb{P}v||_{L^\infty}$ and $||u||_{BMO}=||\mathbb{P}v||_{BMO}$.

We first show $||\sqrt{T-t}\nabla v||_{L^3(\mathbb{R}^3)}$ is bounded. Direct computations show that
\begin{align*}
&\int_{\mathbb{R}^3}|\sqrt{T-t}\nabla v|^3dx\\\leq& C(\int_{\mathbb{R}^3}(T-t)^{\frac{3}{2}}|x|^{3\sqrt{T-t}-3}(-\ln|x|)^{3s}\phi^3(x)dx+\int_{\mathbb{R}^3}|x|^{3\sqrt{T-t}-3}(-\ln|x|)^{3s-3}\phi^3(x)dx\notag\\&+\int_{\mathbb{R}^3}|x|^{3\sqrt{T-t}}(-\ln|x|)^{3s}|\nabla\phi|^3dx)\notag\\
\leq& C\frac{4\pi}{3}((T-t)^{\frac{3}{2}}\int_0^\frac{3}{4} r^{3\sqrt{T-t}-1}(-\ln r)^{3s}dr+\int_0^\frac{3}{4}
r^{3\sqrt{T-t}-1}(-\ln r)^{3s-3}dr+\int_0^\frac{3}{4} r^{3\sqrt{T-t}+2}(-\ln r)^{3s}dr)\notag\\
\leq& C(I_1+I_2+I_3)\notag.
\end{align*}

For $I_1$, it holds by setting $\eta=-3\sqrt{T-t}\ln r$ that
\begin{align}\label{I_1}
I_1=&(T-t)^{\frac{3}{2}}\int_{0}^{\frac{3}{4}}r^{3\sqrt{T-t}-1}(-\ln r)^{3s}dr\\
=&(T-t)^{\frac{3}{2}}\int_{3\sqrt{T-t}\ln\frac{4}{3}}^{\infty}e^{-\eta}(\frac{\eta}{3\sqrt{T-t}})^{3s}\frac{d\eta}{3\sqrt{T-t}}\notag\\
\leq& 3^{-3s-1}(T-t)^{\frac{3-3s-1}{2}}\int_{0}^{\infty}e^{-\eta}\eta^{3s}d\eta\notag\\
=&3^{-3s-1}(T-t)^{\frac{3-3s-1}{2}}\Gamma(3s+1)\notag.
\end{align}
For $I_2$, we have that
\begin{align}\label{I_2}
I_2=&\int_{0}^{\frac{3}{4}}r^{3\sqrt{T-t}-1}(-\ln r)^{3s-3}dr\\
\leq &\int_{0}^{\frac{3}{4}}r^{-1}(-\ln r)^{3s-3}dr\notag\\
\leq &\frac{1}{2-3s}(\ln 4-\ln 3)^{3s-2}.\notag
\end{align}

It is obviously that
\begin{align}\label{I_3}
I_3=\int_0^\frac{3}{4} r^{3\sqrt{T-t}+2}(-\ln r)^{3s}dr\leq C.
\end{align} Combining with the estimates( \ref{I_1})-(\ref{I_3}), we obtain
\begin{align}
\int_{\mathbb{R}^3}|\sqrt{T-t}\nabla v|^3dx\leq C
\end{align} for $0<s<\frac{2}{3}$.

Using the above estimate and the boundedness of $\mathbb{P}$ on $L^3$, we get
\begin{align}
||\sqrt{T-t}\nabla u(t)||_{L^3}=&||\sqrt{T-t}\mathbb{P}\nabla v||_{L^3}\leq C_1||\sqrt{T-t}\nabla v||_{L^3}\leq C.\notag
\end{align}
We thus obtain by using the fact \.{W}$^{1,3}(\mathbb{R}^3)\hookrightarrow BMO(\mathbb{R}^3)$ that
\begin{align}
||u(t)||_{BMO(\mathbb{R}^3)}\leq\frac{C}{\sqrt{T-t}}.
\end{align}

We now consider the estimate of $||u(t)||_{L^\infty}$. By the definition of Leray projection, we see that
\begin{align}
u_i=(\mathbb{P}v)_i=v_{i}-\mathbb{R}_i\mathbb{R}_jv_{j}
\end{align}where $\mathbb{R}_i$ is the Riesz transform.

On the one hand, we deduce by choosing $|x|=e^{-\frac{s}{\sqrt{T-t}}}$
\begin{align}
||v(t)||_{L^\infty}\sim\frac{1}{(T-t)^\frac{1+s}{2}}.
\end{align}On the other hand,
by the definition of Riesz transform, we have
\begin{align}
\mathbb{R}_i\mathbb{R}_jv_{j}(0)=&\frac{1}{\sqrt{T-t}}\int_{\mathbb{R}^3}\frac{y_iy_j}{|y|^5}|y|^{\sqrt{T-t}}(-\ln|y|)^{s}\phi_j(|y|)dy\\
=&\frac{1}{\sqrt{T-t}}\int_{S^2}\omega_i \omega_j d\omega\int_{0}^{1}r^{\sqrt{T-t}-1}(-\ln r)^{s}\phi_j(r)dr\notag\\
=&\frac{1}{(T-t)^{\frac{1}{2}}}\int_{S^2}|\omega_i|^2d\omega\int_{0}^{1}r^{\sqrt{T-t}-1}(-\ln r)^{s}\phi_j(r)dr\notag\\
\geq&\frac{1}{(T-t)^{\frac{1}{2}}}\int_{S^2}|\omega_i|^2d\omega\int_{0}^{\frac{1}{2}}r^{\sqrt{T-t}-1}(-\ln r)^{s}dr\notag\\
=&\frac{1}{(T-t)^{1+\frac{s}{2}}}\frac{4\pi}{9}\int_{\sqrt{T-t}\ln2}^\infty e^{-\tau}\tau^{s}d\tau\notag\\
\geq& \frac{2\pi}{9}\frac{1}{(T-t)^{1+\frac{s}{2}}}\Gamma(1+s).\notag
\end{align} Collecting (3.21)-(3.23) implies
\begin{align}
||u(t)||_{L^\infty}\geq\frac{1}{(T-t)^{1+\frac{s}{2}}}.
\end{align} Combining with (3.15), (3.20), (3.24) completes the proof of Proposition \ref{prop}.
\end{proof}
\bigskip
\noindent\textbf{Acknowledgments}\textit{\ Tan was supported by the Construct Program of the Key Discipline in Hunan Province and the National Natural Science Foundation of China (No. 11871209).
Yin was partially supported by the National Natural Science Foundation of China (No. 11671407), the Macao Science and Technology Development Fund (No. 0091/2018/A3), and Guangdong Province of China Special Support Program (No. 8-2015), and the key project of the Natural Science Foundation of Guangdong province (No. 2016A030311004). }

\end{document}